\documentclass{amsart}

\usepackage{graphicx}
\usepackage[all]{xypic}

\newtheorem{theorem}{Theorem}[section]
\newtheorem{lemma}[theorem]{Lemma}
\newtheorem{proposition}[theorem]{Proposition}
\newtheorem{corollary}[theorem]{Corollary}

\theoremstyle{definition}
\newtheorem{definition}[theorem]{Definition}
\newtheorem{example}[theorem]{Example}
\newtheorem{exo}[theorem]{Exercise}

\theoremstyle{remark}
\newtheorem{remark}[theorem]{Remark}

\numberwithin{equation}{section}

\newcommand{\Spec}{\textup{Spec}}
\newcommand{\Aut}{\textup{Aut}}
\newcommand{\Hom}{\textup{Hom}}
\newcommand{\Gal}{\textup{Gal}}
\newcommand{\id}{\textup{id}}
\newcommand{\Ker}{\textup{Ker}}
\newcommand{\A}{\mathbf{A}}
\newcommand{\Z}{\mathbf{Z}}
\newcommand{\N}{\mathbf{N}}

\newcommand{\F}{\mathbf{F}}
\newcommand{\YExt}{\mathbf{YExt}}
\newcommand{\Yext}{\textup{YExt}}

\newcommand{\Res}{\mathrm{Res}}

\newcommand{\DIAG}{{\mathbf{Diag}}}
\newcommand{\Diag}{\textup{Diag}}

\begin{document}

\title{Splitting families in Galois cohomology}

\author{Cyril Demarche}
\address{Sorbonne Universit\'e, Universit\'e Paris Diderot, CNRS, Institut de Math\'ematiques de Jussieu-Paris Rive Gauche, IMJ-PRG, F-75005, Paris, France}
\curraddr{D\'epartement de math\'ematiques et applications, \'Ecole normale sup\'erieure,
CNRS, PSL Research University, 45 rue d'Ulm, 75005 Paris, France}
\email{cyril.demarche@imj-prg.fr}
\thanks{The first author acknowledges the support of the Agence Nationale de la Recherche under reference ANR-15-CE40-0002}

\author{Mathieu Florence}
\address{Sorbonne Universit\'e, Universit\'e Paris Diderot, CNRS, Institut de Math\'ematiques de Jussieu-Paris Rive Gauche, IMJ-PRG, F-75005, Paris, France}
\email{mathieu.florence@imj-prg.fr}
\thanks{The second author acknowledges the support of the Agence Nationale de la Recherche under reference  GeoLie  (ANR-15-CE40-0012).}


\date{\today}



\begin{abstract}
Let $k$ be a field, with absolute Galois group $\Gamma$. Let $A/k$ be a finite \'etale group scheme of multiplicative type, i.e. a finite discrete $\Gamma$-module. Let $n \geq 2$ be an integer, and let $x \in H^n(k,A)$ be a cohomology class. We show that there exists a countable set $I$, and a familiy  $(X_i)_{i \in I}$ of (smooth, geometrically integral) $k$-varieties, such that the following holds: for any field extension $l/k$, the restriction of  $x$ vanishes in $H^n(l,A)$  if and only if (at least) one of the $X_i$'s has an $l$-point. In addition, we show that the $X_i$'s can be made into an ind-variety. In the case $n=2$, we note that one variety is enough.
\end{abstract}

\maketitle

\section*{Introduction}
Let $k$ be  a field, and let $p$ be a prime number, which is invertible in $k$.
The notion of a norm variety was introduced in the study of the Bloch-Kato conjecture.  It is a key tool in the proof provided by Rost, Suslin and Voevodsky. The norm variety $X(s)$ of a pure symbol $$s=(x_1) \cup (x_2) \cup \ldots \cup (x_n) \in H^n(k, \mu_p^{\otimes n}),$$ where the $x_i$'s are  elements of $k^\times$, was constructed by Rost (cf. \cite{Rost} or \cite{SJ}). The terminology 'norm variety' reflects that it is defined through an inductive process involving the norm of finite field extensions of degree $p$. It has the remarkable property that, if $l/k$ is a field extension, then the restriction of $s$ vanishes in $H^n(l, \mu_p^{\otimes n})$ if and only if the $l$-variety $X(s)_l$ has a $0$-cycle of degree prime-to-$p$. It enjoys nice geometric features, which we will  not mention here. For $n \geq 3$, norm varieties are, to the knowledge of the authors of this paper, known to exist for pure symbols only.\\
In this paper, we shall be interested in the following closely related problem.  Let $A/k$ be a finite \'etale group scheme of multiplicative type, that is to say, a finite discrete $\Gamma$-module. Consider a class $x \in H^n(k,A)$. Does there exists a \textit{countable} family of smooth $k$-varieties $(X_i)_ {i \in I}$, such that, for every field extension $l/k$, the presence of a $l$-point in (at least) one of the $X_i$'s is equivalent to the vanishing of $x$ in $H^n(l, A)$ ?  If such a family exist,  can it always be endowed with the structure of an ind-variety?

We provide  answers to those questions. The main results of the paper are the following:
\begin{theorem}[Corollary \ref{cor cohom deg 2} and Corollary \ref{cor cohom general}]
Let $A/k$ be a finite \'etale group scheme of multiplicative type and let $\alpha \in H^n(k,A)$, where $n \geq 2$ is an integer. 
\begin{itemize}
    \item There exists a countable family $(X_i)_{i \in I}$ of smooth geometrically integral $k$-varieties, such that for any field extension $l/k$ with $l$ infinite, $\alpha$ vanishes in $H^n(l,A)$ if and only if $X_i(l) \neq \emptyset$ for some $i$. In addition, there is such a family $(X_i)$ which is an ind-variety.
    \item If $n=2$,  the family $(X_i)$ can be replaced by a single smooth geometrically integral $k$-variety.
\end{itemize}
\end{theorem}

Note that our main "non-formal" tool, as often (always?) in this context, is Hilbert's Theorem $90$.

\section{Notation and definitions.}
In this paper, $k$ is a field, with a given separable closure $k^s$. We denote by $\Gamma := \Gal(k^s / k)$ the absolute Galois group. The  letters  $d$ and $n$ denote two positive integers. We assume $d$ to be invertible in $k$.

We denote by $\mathcal{M}_d$ the Abelian category of finite $\Z /d \Z$-modules, and by $\mathcal{M}_{\Gamma,d}$  that of finite and discrete $\Gamma$-modules of $d$-torsion. The latter is equivalent to the category of finite $k$-group schemes of multiplicative type, killed by $d$. We denote this category by $\mathcal{M}_{k,d}$. When no confusion can arise, we will identify these categories without further notice. We have an obvious forgetful functor $\mathcal{M}_{\Gamma,d} \to \mathcal{M}_d$. 

\subsection{Groups and cohomology.}

Let $G$ be a linear algebraic $k$-group; that is, an affine $k$-group scheme of finite type. We denote by $H^1(k,G)$ the set of isomorphism classes of $G$-torsors, for the fppf topology. It coincides with the usual Galois cohomology set if $G$ is smooth. Let $\varphi: H \to G$ be a morphism of linear algebraic $k$-groups. It induces, for every field extension $l/k$, a natural map $ H^1(l, H) \to H^1(l, G),$ which we denote by $\varphi_{l,*}$.

\subsection{Yoneda Extensions.}

Let $\mathcal{A}$ be an Abelian category.
For all $n \geq 0$, $A, B \in \mathcal{A}$,  we denote by $\YExt^n_{\mathcal{A}}(A,B)$ (or $\YExt^n(A,B)$) the (additive) category of Yoneda $n$-extensions of $B$ by $A$, and by $\Yext^n_{\mathcal{A}}(A,B)$ (or $\Yext^n(A,B)$) the Abelian group of Yoneda equivalence classes in $\YExt^n(A,B)$ .

\begin{remark}
The groups $\Yext^n_{\mathcal{A}}(A,B)$ can also be defined as $\Hom_{\mathcal{D}(\mathcal{A})}(A, B[n])$, where $\mathcal{D}(\mathcal{A})$ denotes the derived category of $\mathcal{A}$.
\end{remark}

Given $A, B \in \mathcal{M}_d$, we put $\YExt^n_{d}(A,B) := \YExt^n_{\mathcal{M}_d}(A,B)$. Given $A, B \in \mathcal{M}_{k,d}$, we  put  $\YExt^n_{k,d}(A,B) := \YExt^n_{\mathcal{M}_{k,d}}(A,B)$.

\begin{remark}\label{compcoh}
Let $A$ be a finite discrete $\Gamma$-module.\\
Let $d$ be the exponent of $A$. Then there is a canonical isomorphism 
$$\Yext^n_{k,d}(\Z / d \Z, A) \xrightarrow{\sim} H^n(\Gamma, A) \, $$
 where $H^n(\Gamma,A)$ denotes the usual $n$-th cohomology group.
\end{remark}

\begin{remark}

Let $l/k$ be any field extension.  For $A, B \in \mathcal{M}_{k,d}$, we have a restriction map $$ \mathrm{Res_{l/k}}: \YExt^n_{k,d}(A,B) \longrightarrow \YExt^n_{l,d}(A,B) .$$ 
\end{remark}

\subsection{Lifting triangles.}

Let $\varphi : H \to G$ be a morphism of linear $k$-algebraic groups.
 A lifting triangle (relative to $\varphi$) is a commutative triangle 

\[
\xymatrix{T:
Q \ar[r]^f \ar[dr]_{H_X} & P \ar[d]^{G_X} \\
& X \, ,
}
\]

where $X$ is a $k$-scheme, $Q \longrightarrow X$ (resp. $P \longrightarrow X$) is an $H_X$-torsor (resp. a $G_X$-torsor), and where $f$ is an $H$-equivariant morphism (formula on the functors of points: $ f(h.x)=\varphi(h).f(x)$).\\ Note that such a diagram  is equivalent to the data of an isomorphism between  the $G_X$-torsors $P$ and $\varphi_*(Q)$.\\
The $k$-scheme $X$ is called the \textit{base} of the lifting triangle $T$.\\ We have an obvious notion of isomorphism of lifting triangles. \\Moreover, if $\eta: Y \longrightarrow X$ is a morphism of $k$-schemes, we can form the pullback $\eta^*(T)$; it is a lifting triangle, over the base $Y$.

\subsection{Lifting varieties.}

Let $\varphi : H \to G$ be a morphism of linear $k$-algebraic groups.
Let $P \to \Spec(k)$ be a torsor under the group $G$.\\

A geometrically integral $k$-variety $X$ will be called a \textit{lifting variety} (for the pair $(\varphi, P)$) if it fits into a lifting triangle $\mathcal T$:
\[
\xymatrix{
\mathcal{Q} \ar[r]^F \ar[dr]_{H_X} & P \times_k X \ar[d]^{G_X} \\
& X \, ,
}
\]

such that the following holds:\\

For every field extension $l/k$, with $l$ infinite, and for every lifting triangle $t$:  \[
\xymatrix{
Q \ar[r]^f \ar[dr]_{H_l} & P \times_k l \ar[d]^{G_l} \\
& \Spec(l) \, ,
}
\] the set of $l$-rational points $x : \Spec(l) \longrightarrow X$ such that the pullback $\mathcal T_x:=x^*(\mathcal T)$ is isomorphic to $t$ (as a lifting triangle over $\Spec(l)$) is Zariski-dense in $X$, hence non-empty.\\

In particular, the variety $X$ has an $l$-point if and only if the class of the $G$-torsor $P$ in $H^1(l, G)$ is in the image of the map $\varphi_{l,*} : H^1(l,H) \to H^1(l,G)$.

\subsection{Splitting families.}

Let  $A, B$ be objects of $\mathcal{M}_{k,d}$.  Pick a class $x \in \YExt^n_{k,d}(A,B)$.\\

A countable set $(X_i)_ {i \in I}$ of (smooth, geometrically integral) $k$-varieties will be called a \textit{splitting family} for $x$ if the following holds:\\

For every field extension $l/k$, with $l$ infinite, $\Res_{l/k}(x) $ vanishes in $\YExt^n_{l,d}(A,B)$ if and only if (at least) one of the $l$-varieties $X_i$ possesses a $l$-point. \\

Whenever a splitting family exists, it is natural to ask whether it can be made into an ind-variety. By this, we mean here that $I=\mathbb N$ and that, for each $i \geq 0,$ we are given a closed embedding of $k$-varieties $X_i \longrightarrow X_{i+1}$.


\section{Existence of lifting varieties.}

This section contains the non-formal ingredient of this paper, which may have an interest on its own.\\

Let $\varphi : H \to G$ be a morphism of linear $k$-algebraic groups; that is, of affine $k$-group schemes of finite type.\\
Let $P \to \Spec(k)$ be a torsor under the group $G$.\\

The aim of this section is to construct a lifting variety for $(\varphi, P)$. Equivalently, we will build a "nice" $k$-variety $X$ that is a versal object for $H$-torsors that lift the $G$-torsor $P$, in the sense explained in the previous paragraph. 



In particular, recall that $X(l) \neq \emptyset$ if and only if $[P_l]$ lifts to $H^1(l,H)$, for every field extension $l/k$, with $l$ infinite.

To construct such an $X$, we mimick the usual construction of versal torsors (see for instance \cite{Ser}, section I.5). We just have to push it slightly further.\\

There exists a finite dimensional $k$-vector space $V$  endowed with a generically free linear action of $H$. There exists a dense open subset $V_0 \subset \A(V)$,  stable under the action of $H$, and such that  the geometric quotient $$ V_0 \longrightarrow V_0/H$$ exists, and is an $H$-torsor, which we denote by $\mathcal Q$. \\

Form the quotient $$X_{\varphi,P}:=(P \times_k V_0)/H, $$ where $H$ acts on $P$ via $\varphi$, and on $V_0$ in the natural way. Projecting onto $V_0$ induces a  morphism $$\pi: X_{\varphi,P}  \longrightarrow V_0/H, $$ which can also be described as the twist of $P$ by the $H$-torsor $\mathcal Q$, over the base $V_0/H$.\\


Note that $X_{\varphi, P}$ depends on the choice of $V$ (up to stable birational equivalence).\\

If we denote by $\mathcal{Q}'$ the pullback via $\pi$  of the $H$-torsor $\mathcal Q$, there is a natural lifting triangle $\mathcal{T}_{\varphi, P}$:
\[
\xymatrix{
\mathcal{Q}' \ar[rd]_{H} \ar[r] &  X_{\varphi, P} \times_k P \ar[d]^{G} \\
& X_{\varphi, P} 
}
\]
Its existence is explained by the following key fact. If $Y := V_0/H$, then for any $Y$-scheme $S$, a point  $$s \in X_{\varphi, P}(S)=\Hom_{Y-\mathrm{sch}}(S,X_{\varphi, P})$$ is exactly the same as an $H$-equivariant morphism between $\mathcal Q \times_Y S$ and $X_{\varphi, P} \times_k S$, over the base $S$ (see for instance \cite{Gir}, th\'eor\`eme III.1.6.(ii)), i.e. it is the same as a lifting triangle relative to $\varphi$ over the base $S$, i.e. an isomorphism of $G$-torsors between $\varphi_* \mathcal{Q} \times_Y S$ and $P \times_k S$. We shall refer to this property as the universal property of $X_{\varphi, P}$. 

\begin{proposition} \label{lem versal lift}
The $k$-variety $X_{\varphi, P}$ is a lifting variety for the pair $(\varphi, P)$.


In particular, $X_{\varphi, P}(l) \neq \emptyset$ if and only if $[P_l]$ lifts to $H^1(l,H)$.
\end{proposition}

\begin{proof}
Let $l/k$ be a field extension with $l$ infinite.
Let 
\[
\xymatrix{t:
Q \ar[r]^f \ar[dr]_{H_l} & P \times_k l \ar[d]^{G_l} \\
& \Spec(l) \, 
}
\]
be a lifting triangle, over $l$. By Hilbert's Theorem 90 (for $\mathbf{GL}_k(V)$), the set of $l$-rational points $$x \in (V_0/H)(l)= \Hom_{k-\mathrm{sch}}(\Spec (l), V_0/H)$$ such that $x^*(\mathcal Q)$ is isomorphic to $Q$ (as $H$-torsors over $l$)  is Zariski-dense. Let $x$ be such a point.  Then the lifting triangle $t$  corresponds to an isomorphism of $G$-torsors between $\varphi_*( Q )$ and $P$, over the base $\Spec(l)$. Since $Q$ is isomorphic to $x^*(\mathcal Q)$, the universal property of $X_{\varphi, P}$  implies that the lifting triangle $t$ is isomorphic to the fiber of $\mathcal{T}_{\varphi, P}$ at an $l$-rational point of $X_{\varphi, P}$. This finishes the proof. 



\end{proof}

\begin{lemma} \label{lem smooth rat}
The $k$-variety $X_{\varphi, P}$ is smooth and geometrically unirational if $\varphi : H \to G$ is surjective, or if $G$ is smooth and connected.
\end{lemma}

\begin{proof}
To prove this, we can assume that  $k=\bar k$,  in which case the torsor $P$ is trivial. Then $X_{\varphi, P}=(G \times V_0)/H$ . If $G$ is smooth and connected, then it is $k$-rational. Hence  $G \times V_0$ is smooth, connected and $k$-rational as well. The  quotient morphism $$ G \times V_0  \longrightarrow X_{\varphi, P}$$ is an $H$-torsor, and smoothness and geometrical unirationality of its total space implies that of its base.\\
Now, assume that $\varphi$ is surjective. Denoting by $K$ its kernel, we see that $X_{\varphi, P}=V_0 /K$,  which implies the result. 
\end{proof}

\section{Triviality of Yoneda extensions in Abelian categories.}

Let $\mathcal{A}$ be an Abelian category.

The following lemma is well-known.

\begin{lemma} \label{lem ext triv}
Let $\mathcal{E} = (0 \to B \xrightarrow{f_0} E_1 \xrightarrow{f_1} \dots \to E_{n-1} \xrightarrow{f_{n-1}} E_n \xrightarrow{f_n} A \to 0)$ be an object in $\YExt^n(A,B)$, and let $e$ denote its class in $\Yext^n(A,B)$. 

Then $e = 0$ in $\Yext^n(A,B)$ if and only if there exists $\mathcal{F}$ in $\YExt^{n-1}(E_n,B)$ and a morphism of complexes $\phi : \mathcal{E} \to \mathcal{F}$ inducing the identity on $B$ and $E_n$, i.e. a commutative diagram (with exact rows)
\begin{equation} \label{diag split}
\xymatrix{
0 \ar[r] & B \ar[r]^{f_0} \ar[d]^{\id} & E_1 \ar[r]^{f_1} \ar[d]^{\phi_1} & \dots \ar[r] & E_{n-1} \ar[r]^{f_{n-1}} \ar[d]^{\phi_{n-1}} & E_n \ar[d]^{\id} \ar[r]^{f_n} & A \ar[r] & 0 \\
0 \ar[r] & B \ar[r] & F_1 \ar[r]^{g_1} & \dots \ar[r] & F_{n-1} \ar[r]^{g_{n-1}} & E_n \ar[r] & 0 & \, . 
}
\end{equation}
\end{lemma}

\begin{proof}
By \cite{Oort}, section 2 (see also \cite{Bou}, section 7.5, theorem 1, in the case of categories of modules), $e = 0$ is and only if there exists a commutative diagram
\begin{equation} \label{diag ext triv}
\xymatrix{
0 \ar[r] & B \ar[r] \ar[d] & E_1 \ar[r] \ar[d] & \dots \ar[r] & E_{n-1} \ar[r] \ar[d] & E_n \ar[r] \ar[d] & A \ar[r] \ar[d]^= & 0 \\
0 \ar[r] & B \ar[r] & G_1 \ar[r] & \dots \ar[r] & G_{n-1} \ar[r] & G_n \ar[r] & A \ar[r] & 0 \\
0 \ar[r] & B \ar[r]^\id \ar[u] & B \ar[r] \ar[u] & 0 \dots 0 \ar[r] & 0 \ar[r] \ar[u] & A \ar[r]^\id \ar[u] & A \ar[r] \ar[u]^= & 0 \, . \\
}
\end{equation}

Assume $e = 0$. In the previous diagram, let $K' := \Ker(G_n \to A)$. Since we are given a splitting $s$ of $G_n \to A$, there is a natural map $E_n \to K'$ defined via the retraction of $K' \to G_n$ associated to $s$. Define $\mathcal{F}$ to be the pull-back of the exact sequence 
$$0 \to B \to G_1 \to \dots \to G_{n-1} \to K' \to 0$$
by the aforementionned morphism $E_n \to K'$. It is now clear that $\mathcal{F}$ satisfies the statement of the Lemma.

To prove the converse, assume the existence of $\mathcal{F}$ and $\phi$ as in the Lemma. Define $F_i := G_i$ for all $i \leq n-1$, and $G_n := E_n \oplus A$. Consider the maps $h_i := g_i$ for $i \leq n-2$, and let $h_{n-1} :=g_{n-1} \oplus 0 : G_{n-1} \to G_n = E_n \oplus A$ and $h_n : G_n = E_n \oplus A \to A$ be the natural projection. Then the morphism $\phi$ together with the map $\id \oplus f_n : E_n \to G_n = E_n \oplus A$ defines a commutative diagram of the shape \eqref{diag ext triv}, hence $e=0$.
\end{proof}


\begin{definition}
Given $\mathcal{E} \in \YExt^n(A,B)$ as in Lemma \ref{lem ext triv}, a $\mathcal{E}$-diagram is a pair $(\mathcal{F}, \phi)$, where $\mathcal{F} \in \YExt^{n-1}(E_n,B)$ and $\phi : \mathcal{E} \to \mathcal{F}$ is a morphism of complexes inducing the identity on $B$ and $E_n$ (see diagram \eqref{diag split}). Such a diagram is called injective if $\phi_i$ is a monomorphism for all $i$.
\end{definition}

We denote by $\DIAG(\mathcal{E})$ (or $\DIAG_\mathcal{A}(\mathcal{E})$) the category of $\mathcal{E}$-diagrams, where a morphism between $(\mathcal{F}, \phi)$ and $(\mathcal{F}', \phi')$ is a morphism between the commutative diagrams associated (as in Lemma \ref{lem ext triv}) to both $\mathcal{E}$-diagrams, and by $\Diag(\mathcal{E})$ the set of isomorphism classes in $\DIAG(\mathcal{E})$.

Note that, given $\mathcal{D} = (\mathcal{F}, \phi) \in \DIAG(\mathcal{E})$, there is a natural group homomorphism $\Aut(\mathcal{D}) \to \Aut(\mathcal{E})$. 

\begin{example} \label{ex mkd}
Consider the particular case when $\mathcal{A}$ is the category $\mathcal{M}_{k,d}$. 
Recall the obvious functor $\mathcal{M}_{k,d} \to \mathcal{M}_d$.

Then an object $\mathcal{E}$ of the category $\YExt^n_{k,d}(A,B)$ is exactly the same as an object $\mathcal{E}'$ in $\YExt^n_{d}(A,B)$ together with a (continuous) group homomorphism $p : \Gamma \to \Aut(\mathcal{E}')$.

Moreover, a $\mathcal{E}$-diagram $\mathcal{D}$ in the category $\mathcal{M}_{k,d}$ is the same as a $\mathcal{E}'$-diagram $\mathcal{D}'$ in the category $\mathcal{M}_d$ together with a homomorphism $q : \Gamma \to \Aut(\mathcal{D}')$ lifting $p$.

Note that in this context, the groups $\Aut(\mathcal{D}')$ and $\Aut(\mathcal{E}')$ are finite.
\end{example}

\section{Splitting varieties for $2$-extensions}

In this section, we restrict to the special case of $\Yext^2_k(A,B)$ and
$H^2(k,A)$, and we construct splitting varieties.

\begin{theorem} \label{thm deg 2}
Let $A, B$ be a finite $d$-torsion $\Gamma$-modules and $e \in \Yext^2_{k,d}(A,B)$.

Assume $A$ or $B$ is free as a $\F$-module.

Then, there exists a smooth geometrically integral $k$-variety $X$ which is a splitting variety for $e$.
\end{theorem}

\begin{proof}
Let $\mathcal{E} = (0 \to B \to E_1 \to E_2 \to A \to 0)$ be a $2$-extension of $d$-torsion $\Gamma$-modules representing $e$. Using Pontryagin duality $\Hom(., \F)$, one can assume $B$ is free. Lemma \ref{lem E_n free} below implies that one can also assume that $E_2$ is free as a $\F$-module.

A $\mathcal{E}$-diagram in $\mathcal{M}_d$ is a commutative diagram with exact lines in the category of finite $d$-torsion abelian groups:
\begin{equation} \label{diag deg 2}
\xymatrix{
0 \ar[r] & B \ar[r] \ar[d]^{\id}_\sim & E_1 \ar[r] \ar[d]^{\phi_1} & E_2 \ar[r] \ar[d]^{\id}_\sim & A \ar[r] & 0 \\
0 \ar[r] & B \ar[r] & F_1 \ar[r] & E_2 \ar[r] & 0 & \, .
}
\end{equation}
In particular, in such a diagram, $F_1$ is free as a $\F$-module. Therefore Lemma \ref{lem equiv diag} below implies that there is a unique such diagram, say $\mathcal{D}$, up to isomorphism.

The $2$-extension $\mathcal{E}$ defines a group homomorphism $p : \Gamma \to \Aut(\mathcal{E}):=\Aut_{\mathcal{M}_d}(\mathcal{E})$ (see example \ref{ex mkd}), so that $p$ corresponds to a $\Spec(k)$-torsor $P_{\mathcal{E}}$ under $\Aut(\mathcal{E})$.

Then Example \ref{ex mkd} relates the triviality of the class $e$ to the existence of a lifting of the torsor $P_{\mathcal{E}}$ to the group $\Aut(\mathcal{D})$.

Let $X$ be the lifting variety $X_{\varphi, P_{\mathcal{E}}}$ for the natural morphism of finite groups $\varphi : \Aut(\mathcal{D}) \to \Aut(\mathcal{E})$, where those groups are considered as constant algebraic $k$-groups.

Then Example \ref{ex mkd} and Proposition \ref{lem versal lift} imply that $X$ is a splitting variety for $e$.
\end{proof}

\begin{corollary}\label{cor cohom deg 2}
Let $A$ be a finite $d$-torsion $\Gamma$-modules and $\alpha \in H^2(k,A)$.

Then, there exists a smooth geometrically integral $k$-variety $X$ which is a splitting variety for $\alpha$.
\end{corollary}

\begin{remark}
This corollary recovers a result of Krashen (see \cite{Kra}).
\end{remark}

\begin{proof}
By Remark \ref{compcoh}, we have a canonical isomorphism $\Yext^2_{k,d}(\F,A) \xrightarrow{\sim} H^2(k,A)$, hence the Corollary is a direct consequence of Theorem \ref{thm deg 2}.
\end{proof}

\begin{exo} (hard)
Let $R/k$ be a central simple algebra, of index $d$. Assume that $d$ is invertible in $k$. In the previous Corollary, take $A$ to be $\mu_d$, the group of $d$-th roots of unity, and take $\alpha \in H^2(k,A)=H^2(k,\mu_d)$ to be the Brauer class of $R$. Show that its splitting variety $X$, as constructed in the proof of Theorem  \ref{thm deg 2}, is stably birational to the Severi-Brauer variety $\mathrm{SB}(R)$.
\end{exo}

\section{Splitting families for $n$-extensions ($n \geq 3$)}

In this section, we prove the main Theorem of the paper (see Theorem \ref{prop norm ext} below).

Let $d \geq 2$ and $\F := \Z / d \Z$.

Let $n \geq 2$ and let $A, B$ be objects of $\mathcal M_{k,d}$. 

Fix a class $e \in \Yext^n_{k,d}(A,B)$.

\begin{lemma} \label{lem E_n free}
There exists a representative $\mathcal{E} = (0 \to B \xrightarrow{f_0} E_1 \xrightarrow{f_1} \dots \to E_{n-1} \xrightarrow{f_{n-1}} E_n \xrightarrow{f_n} A \to 0)$ of $e$ in $\YExt^n_{k,d}(A,B)$ such that $E_n$ is free.
\end{lemma}

\begin{proof}
Given any representative $\mathcal{E'} = (0 \to B \xrightarrow{f'_0} F_1 \xrightarrow{g_1} \dots \to F_{n-1} \xrightarrow{g_{n-1}} F_n \xrightarrow{g_n} A \to 0)$ of $e$ in $\YExt^n_{k,d}(A,B)$, there exists a finite $d$-torsion $\Gamma$-module $E_n$, free as a $\F$-module, together with a $\Gamma$-equivariant surjection $E_n \to F_n$. Define $F'_{n-1}$ to be $F_{n-1} \times_{F_n} E_n$. Then there is a natural commutative diagram in $\YExt^n_{k,d}(A,B)$
\[
\xymatrix{
0 \ar[r] & B \ar[r]^{g_0} \ar[d]^= & F_1 \ar[r]^{g_1} \ar[d]^{=} & \dots \ar[r] & F_{n-2} \ar[r]^{g_{n-2}} \ar[d]^= & F'_{n-1} \ar[r]^{f_{n-1}} \ar[d] & E_n \ar[r]^{f_n} \ar[d] & A \ar[d]^{=} \ar[r] & 0 \\
0 \ar[r] & B \ar[r]^{g_0} & F_1 \ar[r]^{g_1} & \dots \ar[r] & F_{n-2} \ar[r]^{g_{n-2}} & F_{n-1} \ar[r]^{g_{n-1}} & F_n \ar[r]^{g_{n}} & A \ar[r] & 0 \, , 
}
\]
which proves the Lemma.
\end{proof}

\begin{remark}
Repeating the construction of the proof of Lemma \ref{lem E_n free}, one can even assume that $E_2, \dots, E_n$ are free as $\F$-modules.
\end{remark}

We now fix once and for all a $n$-extension $$\mathcal{E} = (0 \to B \xrightarrow{f_0} E_1 \xrightarrow{f_1} \dots \to E_{n-1} \xrightarrow{f_{n-1}} E_n \xrightarrow{f_n} A \to 0)$$ in $\YExt^n_{k,d}(A,B)$ representing $e$ such that:
\begin{itemize}
	\item $E_n$ is free as a $\F$-module.
    \item $\lvert E_{n-1} \rvert$ is minimal for the given $E_n$.
    \item $\lvert E_{n-2} \rvert$ is minimal among representatives of $e$ with minimal $\lvert E_{n-1} \rvert$ and given $E_n$.
    \item $\lvert E_{n-3} \rvert$ is minimal among representatives of $E$ with minimal $\lvert E_{n-1} \rvert$, minimal $\lvert E_{n-2} \rvert$ and given $E_n$.
    \item \dots
    \item $\lvert E_2 \rvert$ minimal among $n$-extensions representing $E$ with minimal $\lvert E_{n-1} \rvert$, \dots, minimal $\lvert E_3 \rvert$ and given $E_n$.
\end{itemize}

Let us define the notion of a free $n$-extension.

\begin{definition}
An object $$\mathcal{L} = (0 \to D \to L_1 \to \dots \to L_{n-1} \to L_n \to C \to 0)$$ in $\YExt^n_{k,d}(C,D)$ is said to be free if $L_i$ is free as a $\F$-module, for all $1 \leq i \leq n$. 
\end{definition}

\begin{lemma} \label{lem free diag}
Assume that $B$ is free as a $\F$-module. Then the class $e$ is trivial in $\Yext^n_{k,d}(A,B)$ if and only if there exists an injective $\mathcal{E}$-diagram $\phi : \mathcal{E} \to \mathcal{L}$, where $\mathcal{L} \in \YExt^{n-1}_{k,d}(E_n,B)$ is free.
\end{lemma}

\begin{proof}
The existence of such a diagram implies the triviality of $e$, by Lemma \ref{lem ext triv}.

Let us now prove the converse. Assume $e = 0$. Then by Lemma \ref{lem ext triv}, there exists a $\mathcal{E}$-diagram $\varphi : \mathcal{E} \to \mathcal{G}$ of the following shape:
\begin{equation} \label{diag inj}
\xymatrix{
0 \ar[r] & B \ar[r]^{f_0} \ar[d]^= & E_1 \ar[r]^{f_1} \ar[d]^{\varphi_1} & \dots \ar[r] & E_{n-1} \ar[r]^{f_{n-1}} \ar[d]^{\varphi_{n-1}} & E_n \ar[d]^= \ar[r]^{f_n} & A \ar[r] & 0 \\
0 \ar[r] & B \ar[r]^{h_0} & G_1 \ar[r]^{h_1} & \dots \ar[r] & G_{n-1} \ar[r]^{h_{n-1}} & E_n \ar[r] & 0 & \, ,
}
\end{equation}

We now prove by induction that all $\varphi_i$ are injective. By construction, $\varphi_n$ is injective.

Assume that $\varphi_i$ is injective for all $k < i \leq n$. Let us prove that $\varphi_k$ is also injective. 

If $k= 1$, this is elementary.


Assume now that $k \geq 2$.
Consider the quotient $\overline{E}_k := E_k / \Ker(\varphi_k)$, define $\overline{E}_{k-1} := G_{k-1} \times_{G_k} \overline{E}_k$. Then we have a natural commutative diagram of $n$-extensions of $A$ by $B$:
\[
\xymatrix{
 \dots \ar[r] & E_{k-3} \ar[d] \ar[r] & E_{k-2} \ar[r] \ar[d] & E_{k-1} \ar[r] \ar[d] & E_k \ar[r] \ar[d] & E_{k+1} \ar[r] \ar[d]^{=} & \dots \\
\dots \ar[r] & G_{k-3} \ar[d]^= \ar[r] & G_{k-2} \ar[r] \ar[d]^= & \overline{E}_{k-1} \ar[r] \ar[d] & \overline{E}_k \ar[r] \ar[d] & E_{k+1} \ar[r] \ar[d] & \dots \\
\dots \ar[r] & G_{k-3} \ar[r] & G_{k-2} \ar[r] & G_{k-1} \ar[r] & G_k \ar[r] & G_{k+1} \ar[r] & \dots \, .
}
\]
In particular, the $n$-extension 
$$0 \to B \xrightarrow{h_0} G_1 \xrightarrow{h_1} \dots \to G_{k-3} \xrightarrow{h_{k-3}} G_{k-2} \to \overline{E}_{k-1} \to \overline{E}_k \to E_{k+1} \to \dots \to E_n \xrightarrow{f_n} A \to 0$$
represents the class $e$, with $\lvert \overline{E}_k \rvert \leq \lvert E_k \rvert$. By minimality of the extension $\mathcal{E}$, we have $\lvert \overline{E}_k \rvert = \lvert E_k \rvert$, hence $\varphi_k$ is injective.

Hence we proved the existence of an injective diagram $\varphi : \mathcal{E} \to \mathcal{G}$ (see \eqref{diag inj}).

It is now sufficient to prove the existence of an injective morphism $\phi' : \mathcal{G} \to \mathcal{L}$ in $\YExt^{n-1}_{k,d}(E_n, B)$, with $\mathcal{L}$ free.

The $\Gamma$-module $G_1$ can be embedded in a finite $\Gamma$-module $L_1$ that is free as a $\F$-module. Then we have a natural commutative diagram:
\[
\xymatrix{
0 \ar[r] & B \ar[r]^{h_0} \ar[d]^{=} & G_1 \ar[r]^{h_1} \ar[d]^{\phi'_1} & G_2 \ar[r]^{h_2} \ar[d]^{\widetilde{\phi_2}} & G_3 \ar[r]^{h_3} \ar[d]^= & \dots \ar[r] & G_{n-1} \ar[r]^{h_{n-1}} \ar[d]^{=} & E_n \ar[r] \ar[d]^= & 0 & \\
0 \ar[r] & B \ar[r]^{m_0} & L_1 \ar[r]^{h'_1}  & G'_2 \ar[r] & G_3 \ar[r]^{h_3} & \dots \ar[r] & G_{n-1} \ar[r]^{h_{n-1}} & E_n \ar[r] & 0 & \, ,
}
\]
where $\phi_1'$ and $\widetilde{\phi_2}$ are injective and $L_1$ is free. An easy induction (starting by embedding $G'_2$ into a $\Gamma$-module that is free as a $\F$-module) proves that there exists a commutative diagram
\[
\xymatrix{
0 \ar[r] & B \ar[r]^{h_0} \ar[d]^{=} & G_1 \ar[r]^{h_1} \ar[d]^{\phi'_1} & G_2 \ar[r]^{h_2} \ar[d]^{\phi'_2} & \dots \ar[r] & G_{n-1} \ar[r]^{h_{n-1}} \ar[d]^{\phi'_{n-1}} & E_n \ar[r] \ar[d]^= & 0 & \\
0 \ar[r] & B \ar[r]^{m_0} & L_1 \ar[r]^{m_1}  & L_2 \ar[r]^{m_2} & \dots \ar[r] & L_{n-1} \ar[r]^{m_{n-1}} & E_n \ar[r] & 0 & \, ,
}
\]
with all vertical maps injective, and $L_1, \dots, L_{n-2}$ free as $\F$-modules. Since $E_n$ is free as a $\F$-module (see Lemma \ref{lem E_n free}), then $L_{n-1}$ is also free as a $\F$-module, which concludes the proof.
\end{proof}

For any non-negative integers $a, b, m$, define $\mathcal{E}_{n-1}(a,b,m)$ to be the following $(n-1)$-extension of free $\F$-modules:
$$\mathcal{E}(a,b,m) := (0 \to F_0 \xrightarrow{g_0} F_1 \xrightarrow{g_1} \dots \to F_{n-2} \xrightarrow{g_{n-2}} F_{n-1} \xrightarrow{g_{n-1}} F_{n} \to 0) \, ,$$
where $F_0 := \F^b$, $F_1 := \F^b \oplus \F^m$, $F_2 = \dots = F_{n-2} = \F^{m} \oplus \F^m$, $F_{n-1} = \F^{m} \oplus \F^a$,  $F_{n} = \F^a$, and $g_0(x) := (x,0)$, $g_i(x,y) = (y,0)$ for $1 \leq i \leq n-2$ and $g_{n-1}(x,y) := y$.

\begin{lemma}
\label{lem ext can}
Assume $B$ is free as a $\F$-module. Let $\mathcal{L} = (0 \to B \xrightarrow{m_0} L_1 \xrightarrow{m_1} \dots \to L_{n-1} \xrightarrow{m_{n-1}} E_n \to 0)$ be an object in $\YExt^{n-1}_{k,d}(E_n,B)$ that is free. Let $a$ (resp. $b$) be the rank of $E_n$ (resp. $B$). 

Then there exist an integer $m$, a Galois action on $\mathcal{E}_{n-1}(a,b,m)$ and an injective morphism $\phi : \mathcal{E} \to \mathcal{E}_{n-1}(a,b,m)$ in $\YExt^{n-1}_{k,d}(E_n,B)$.
\end{lemma}

\begin{proof}
Choose $m$ large enough such that $m$ is greater than or equal to the rank of $L_i$, for all $i$. 

Splitting the $(n-1)$-extension into short exact sequences, the statement reduces to two facts :
\begin{itemize}
	\item given any free $\F$-module $L$ with a $\Gamma$-action and any integer $s$ greater than or equal to the rank of $L$, there exists a decomposition $\F^s = L \oplus L'$ and therefore a $\Gamma$-action on $\F^s = L \oplus L'$ such that $\Gamma$ acts trivially on $L'$, with a $\Gamma$-equivariant embedding of $L$ into $\F^s$.
	\item given a diagram of short exact sequences of free $\F$-modules (the second one being the obvious one)
\[
\xymatrix{
0 \ar[r] & L_1 \ar[d] \ar[r] & L_2 \ar[r] & L_3 \ar[r] \ar[d] & 0 \\
0 \ar[r] & \F^r \ar[r]^i & \F^r \oplus \F^s \ar[r]^p & \F^s \ar[r] & 0
}
\]
where the vertical maps are injective, $L_1, L_2, L_3$, $\F^r$ and $\F^s$ are endowed with a $\Gamma$-action such that the arrows are $\Gamma$-equivariant, and assuming there is a decomposition $\F^s = L_3 \oplus L'$ such that the action on $L'$ is trivial, there exists a $\Gamma$-action on $\F^r \oplus \F^s$ and an embedding $L_2 \to \F^r \oplus \F^s$ making the previous diagram a commutative diagram of $\Gamma$-modules.

Indeed, the choice of a section of the first line and the action of $\Gamma$ on $L_2$ define a map $\rho : \Gamma \to \Hom(L_3,L_1)$ satisfying a cocycle condition
$$\rho(\sigma \tau)(x) = {^\sigma \rho(\tau)(x)} + \rho(\sigma)(^{\tau} x) \, .$$

In order to prove the aforementioned fact, one needs to extend $\rho$ to a map $\tilde{\rho} : \Gamma \to \Hom(\F^s, \F^r)$ satisfying a similar condition. One easily checks that the maps $\tilde{\rho}(\gamma) : \F^s = L_3 \oplus L' \to \F^r$ defined by $\widetilde{\rho}(\gamma)(x,y) := \rho(\gamma)(x)$ do satisfy this condition.
\end{itemize}
\end{proof}

\begin{lemma} \label{lem equiv diag}
Assume $B$ is free as a $\F$-module. Let $\phi, \psi : \mathcal{E} \to \mathcal{F}$ be two injective $\mathcal{E}$-diagrams in the category $\mathcal{M}_d$, such that $\mathcal{F}$ is free.

Then there exists an automorphism $\epsilon : \mathcal{F} \to \mathcal{F}$ in $\YExt^{n-1}_{d}(E_n,B)$ such that $\psi = \epsilon \circ \phi$.
\end{lemma}

\begin{proof}
As in the proof of Lemma \ref{lem ext can}, splitting the $n$-extension $\mathcal{E}$ into short exact sequences reduces the statement to the following facts:
\begin{itemize}
	\item given two embeddings of $\F$-modules $\phi : E \to F$ and $\psi : E \to F$ with $F$ free, there exists an automorphism $\epsilon$ of $F$ such that $\psi = \epsilon \circ \phi$. Indeed, one only needs to choose one basis of $F$ adapted to each embedding.
	\item given two diagrams of short exact sequences of $\F$-modules
\[
\xymatrix{
0 \ar[r] & A_1 \ar[d]^{\phi_1, \psi_1} \ar[r] & A_2 \ar[d]^{\phi_2, \psi_2} \ar[r] & A_3 \ar[r] \ar[d]^{\phi_3, \psi_3} & 0 \\
0 \ar[r] & F_1 \ar[r]^i & F_2 \ar[r]^p & F_3 \ar[r] & 0
}
\]
where the vertical maps are injective and the $F_i$ are free, and given $\epsilon_1 \in \Aut(F_1)$ and $\epsilon_3 \in \Aut(F_3)$ such that $\psi_i = \epsilon_i \circ \phi_i$, there exists $\epsilon_2 \in \Aut(F_2)$ such that $\psi_2 = \epsilon_2 \circ \phi_2$ and $(\epsilon_i)_{1 \leq i \leq 3}$ is an automorphism of the bottom exact sequence.

Indeed, the modules $F_i$ being free, one can first fix a section $F_2 = F_1 \oplus F_3$. Then $\phi_2$ and $\psi_2$ induce morphisms $\phi_{2,1}, \psi_{2,1} \in \Hom(A_2, F_1)$ and $\phi_{2,3}, \psi_{2,3} \in \Hom(A_2, F_3)$. Then the existence of $\epsilon_2$ is equivalent to the existence of $\epsilon \in \Hom(F_3, F_1)$ such that $\psi_{2,1} = \epsilon_1 \circ \phi_{2,1} + \epsilon \circ \phi_{2,3}$. Such a $\epsilon$ exists since the map $\phi_3^* : \Hom(F_3, F_1) \to \Hom(A_3, F_1)$ is onto (because $F_1$ and $F_3$ are free and $\phi_3$ is injective).
\end{itemize}
\end{proof}

The following statement is the main result of this section:


\begin{theorem} \label{prop norm ext}
Let $n \geq 3$ and let $A, B$ be objects of $\mathcal M_{k,d}$ and assume $A$ or $B$ is free in $\mathcal{M}_d$. Pick a class $e \in \Yext^n_{k,d}(A,B)$.

Then, there exists a smooth geometrically integral ind-variety $(X_i)_{i \in \N}$, which is a splitting family for $e$.
\end{theorem}

Before proving this Theorem, we state explicitely the following consequence:





\begin{corollary}\label{cor cohom general}
Let $n \geq 3$, let $A$ be a finite $\Gamma$-module and let $\alpha \in H^n(k,A)$.

Then there exists a smooth geometrically integral ind-variety $(X_i)_{i \in \N}$, which is a splitting family for $\alpha$.
\end{corollary}

\begin{proof}
Combine the previous theorem and remark \ref{compcoh}.
\end{proof}

%





We now focus on the proof of the main theorem.

\begin{proof}[Proof of Theorem \ref{prop norm ext}] 
Using Pontryagin duality $\Hom(\cdot, \F)$, one can assume that $B$ is free.

Let $\mathcal{E}$ be a $n$-extension representing $e$, as given by Lemma \ref{lem E_n free}. Let $e$ (resp. $b$) be the rank of $E_n$ (resp. $B$) as a free $\F$-module.

The $n$-extension $\mathcal{E}$ defines a group homomorphism $p : \Gamma \to \Aut(\mathcal{E}):=\Aut_{\mathcal{M}_d}(\mathcal{E})$ (see example \ref{ex mkd}), so that $p$ corresponds to a $\Spec(k)$-torsor $P_{\mathcal{E}}$ under $\Aut(\mathcal{E})$.

Then Example \ref{ex mkd} relates the triviality of the class $e$ to the existence of a $\mathcal{E}$-diagram $\mathcal{D}$ in the category $\mathcal{M}_d$ together with a lifting of the torsor $P_{\mathcal{E}}$ to the group $\Aut_{\mathcal{M}_d}(\mathcal{D})$.

Lemmas \ref{lem free diag} and \ref{lem ext can} ensure that in order to construct the splitting varieties, it is sufficient to consider only injective diagrams (of $\F$-modules) $\phi : \mathcal{E} \to \mathcal{E}_{n-1}(e,b,m)$, for some $m \in \N$, i.e. diagrams of the following shape (the aforementioned lemmas essentially say that such diagrams are cofinal in the category of diagrams):
\begin{equation} \label{diag m}
\xymatrix{
0 \ar[r] & B \ar[r]^{f_0} \ar[d]^{\phi_0}_\sim & E_1 \ar[r]^{f_1} \ar[d]^{\phi_1} & \dots \ar[r] & E_{n-1} \ar[r]^{f_{n-1}} \ar[d]^{\phi_{n-1}} & E_n \ar[d]^{\phi_n}_{\sim} \ar[r]^{f_n} & A \ar[r] & 0 \\
0 \ar[r] & F_0 \ar[r]^{g_0} & F_1 \ar[r]^{g_1} & \dots \ar[r] & F_{n-1} \ar[r]^{g_{n-1}} & F_n \ar[r] & 0 & \, ,
}
\end{equation}
where all $\phi_i$ are injective.

In addition, Lemma \ref{lem equiv diag} implies that one only needs to consider one such diagram for each $m$ (since such diagrams with the same $m$ are equivalent up to an automorphism of $\mathcal{E}_{n-1}(e,b,m)$).

Therefore, let us fix, for all $m \in \N$ (sufficiently large), one diagram $\mathcal{D}_m$ of the shape \eqref{diag m} in the category of $\F$-modules, in a compatible way: the diagram $\mathcal{D}_{m+1}$ for the integer $m+1$ is obtained from the diagram $\mathcal{D}_m$ associated to $m$ by composing the morphism $\phi_m : \mathcal{E} \to \mathcal{E}(e,b,m)$ with the natural (injective) morphism $\mathcal{E}(e,b,m) \to \mathcal{E}(e,b,m+1)$.

We have thus defined a direct system of diagrams $\mathcal{D}_m$. For all $m$, let $X_m$ denotes the $k$-variety $X_{\Aut(\mathcal{D}_m) \to \Aut(\mathcal{E}), P_{\mathcal{E}}}$ defined in Proposition \ref{lem versal lift}. By functoriality of the construction of these varieties and by the natural (injective) group homomorphisms $\Aut(\mathcal{D}_m) \to \Aut(\mathcal{D}_{m+1})$, we get a direct system of $k$-varieties $X_m$. 

In addition, Lemma \ref{lem equiv diag} implies that the morphisms $\Aut(\mathcal{D}_m) \to \Aut(\mathcal{E})$ are surjective, hence the varieties $X_m$ are smooth and geometrically unirational.

By construction, $(X_m)_{m \in \N}$ is a splitting family for $e$, which concludes the proof.
\end{proof}






\bibliographystyle{amsplain}

\end{document}